\documentclass[a4paper, 12pt]{amsart}

\usepackage{amsmath}
\usepackage{amsfonts}
\usepackage{amssymb}
\usepackage[utf8]{inputenc}
\usepackage{enumerate}
\usepackage[T1]{fontenc}

\newtheorem{thrm}{Theorem}[section]
\newtheorem{col}[thrm]{Corollary}

\newtheorem{fact}[thrm]{Proposition}
\newtheorem{prob}[thrm]{Problem}
\newtheorem{question}[thrm]{Question}

\theoremstyle{definition}
\newtheorem{defin}[thrm]{Definition}
\newtheorem{remark}[thrm]{Remark}
\newtheorem{exa}[thrm]{Example}

\DeclareMathOperator{\vspan}{span}
\DeclareMathOperator{\Ker}{Ker}

\newcommand{\NN}{\mathbb{N}}
\newcommand{\RR}{\mathbb{R}}

\title[On weak and pointwise topologies]{On the weak and pointwise topologies in function spaces II}
\author{Miko\l aj Krupski and Witold Marciszewski}

\address{Institute of Mathematics\\ University of Warsaw\\ \newline Ul. Banacha 2\\02--097 Warszawa\\ Poland }
\email{mkrupski@mimuw.edu.pl}
\email{wmarcisz@mimuw.edu.pl}



\date{\today}
\thanks{The first author was partially supported by the Polish National Science Center research grant UMO-2012/07/N/ST1/03525. The second author was partially supported by the Polish National Science Center research grant DEC-2012/07/B/ST1/03363}

\begin{document}

\begin{abstract}
For a compact space $K$ we denote by $C_w(K)$ ($C_p(K)$) the space of continuous real-valued functions on $K$ endowed with the weak (pointwise)
topology. In this paper we discuss the following basic question which seems to be open:
\textit{Let $K$ and $L$ be infinite compact spaces. Can it happen that $C_w(K)$ and $C_p(L)$ are homeomorphic?}

M.\ Krupski proved that the above problem has a negative answer when $K=L$ and $K$ is finite-dimensional and metrizable. We extend this result to the class of finite-dimensional Valdivia compact spaces $K$.
\end{abstract}

\maketitle

\section{Introduction}
In this paper we continue the research initiated in \cite{Kr}.
For a compact space $K$, the set of all real-valued continuous functions on $K$ equipped with 
the supremum norm is a Banach space which we denote by $C(K)$. One can consider two, other than norm, topologies on the set of continuous functions on $K$: the weak topology
i.e. the weakest topology making all linear functionals continuous and the pointwise topology, i.e. the topology inherited from the product space
$\RR^K$.
Let us denote the latter two topological spaces by $C_w(K)$ and $C_p(K)$ respectively.
Unlike the norm topology, both weak and pointwise topologies are often non-metrizable. Namely, $C_w(K)$ is non-metrizable for any infinite compact $K$ and
$C_p(K)$ is non-metrizable provided $K$ is uncountable. It is easy to see that if $K$ is infinite, then the pointwise topology is strictly weaker than the weak
topology. It is not clear however whether these two topologies are always non-homeomorphic. The following problem was addressed in \cite{Kr}.
\begin{prob}\label{problem1}
Can $C_p(K)$ and $C_w(K)$ be homeomorphic for an infinite compact space $K$?
\end{prob}
One can ask even a more general question.
\begin{prob}\label{problem2}
Can $C_p(K)$ and $C_w(L)$ be homeomorphic for infinite compact spaces $K$ and $L$?
\end{prob}

Let us point out here that 
it may happen that a vector space is equipped with two different but homeomorphic topologies.
The following example illustrates this phenomenon.
\begin{exa}
Consider the set $\sigma=\{x\in\ell_2:x_n=0\text{ for all but finitely many }n\}$. From classical results in Infinite-Dimensional Topology it follows that
the space $(\sigma,\lVert\cdot \rVert)$ (i.e. $\sigma$ equipped with the norm topology inherited from the Hilbert space $\ell_2$) is homeomorphic to $(\sigma,\tau_p)$
(i.e. $\sigma$ equipped with the pointwise (product) topology inherited from $\RR^\omega$).
\end{exa}

It was proved in \cite{Kr} that the answer to Problem \ref{problem1} is in the negative provided $K$ is a metrizable $C$-space. In particular,
this covers the important case of all finite-dimensional metrizable compacta.

In this paper we will give a partial generalizations of this result (see Theorem \ref{faktor} and Remark \ref{rem-gen} below).
We will also show that the answer to both Problem
\ref{problem1} and Problem \ref{problem2} is in the negative provided $K$ or $L$ is scattered. It turns out that if $K$ is scattered, then the spaces $C_p(K)$
and $C_w(K)$ can be topologically distinguished by the Fr\'echet-Urysohn property or the property (B) which we introduce in Section \ref{property_B}.
It is curious that in the case when $K$ is scattered we can distinguish $C_p(K)$ and $C_w(K)$ by a specific (and quite simple)
topological property and we could not do it,
for instance, for the unit interval (see \cite[Problem 1]{Kr}),
though for scattered $K$ the weak and the pointwise topologies seem to be closest to each other; they coincide on bounded subsets.

The paper is organized as follows. In Section 2 we introduce the basic notation and definitions used throughout the paper. In Section 3 we collect some simple observations
concerning homeomorphisms between function spaces equipped with the weak and pointwise topologies. In particular, we show that $C_p(K)$ and $C_w(L)$
are never uniformly homeomorphic. In Section 4 we prove our main result (Theorem \ref{faktor}). The technique which we use in the proof differs form
the technique used in \cite{Kr} and is inspired by some ideas from \cite{Ok} and \cite{Kr1}. Next we derive some corollaries to Theorem \ref{faktor}.
In particular,
we show that if $K$ is finite-dimensional Valdivia compactum then $C_p(K)$ and $C_w(K)$ are not homeomorphic. This result is a consequence of Theorem \ref{faktor}
and the following dichotomy: \textit{Every Valdivia compact space is either scattered or contains a closed uncountable metrizable subspace}
(see Proposition \ref{Valdivia_dychotomia}). In Section 5 we introduce and investigate a certain topological property (we call it the property (B)).
This property was suggested to us
by Taras Banakh as a property which possibly could distinguish topologically spaces $C_p(K)$ and $C_w(K)$.
We will show that it helps to solve Problem \ref{problem1}
if (and only if) $K$ is scattered. However, as we note, there is no need to introduce the property (B) in that case. It is well-known that
if $K$ is scattered then the space $C_p(K)$ is Fr\'echet-Urysohn. On the other hand $C_w(K)$ is never Fr\'echet-Urysohn for infinite $K$.
In Section 5 we shall look at this situation more closely. In Section 6 we collect some further, more general remarks concerning the property (B).

\section{Preliminaries}
All spaces under consideration are assumed to be Tychonoff.
We denote by $\omega$ the set of all non-negative integers, and $\NN=\omega\setminus\{0\}$.
Recall that a topological space $X$ is \textit{Fr\'echet-Urysohn} if for any $A\subseteq X$ and
$x\in \overline{A}$, there is a sequence $(x_n)_{n\in\omega}$ of points from $A$ which converges to $x$.
A space $X$ is \textit{scattered} if no nonempty subset $A\subseteq X$ is dense-in-itself. It is well-known that a compact space $K$ is
not scattered if and only if $K$ can be continuously mapped onto the unit interval $[0,1]$. 

Let $K$ be a compact space.
As usual, we identify the set $C(K)^*$, of all continuous linear functionals on $C(K)$, with
$M(K)$ -- the set of all signed Radon measures on $K$ of finite variation. Using this identification we can equip $M(K)$ with the weak* topology.
For $y\in K$ we denote by $\delta_y\in M(K)$ the corresponding Dirac measure. If $A\subseteq M(K)$ then $\vspan(A)$ is the linear space spanned by $A$, i.e. the minimal linear subspace of $M(K)$ containing
$A$.

The constant function equal to zero (on a given space) is denoted by $\underline{0}$. It will be clear from the context what is the domain of $\underline{0}$.

Recall that sets of the form
$$O(F,\tfrac{1}{m})=\{f\in C_p(K):(\forall x\in F)\;\;\lvert f(x) \rvert<\tfrac{1}{m}\},$$
where $F\subseteq K$ is finite and $m\in \NN$, are basic open neighborhoods of the function equal to zero on $K$ in $C_p(K)$.

Similarly, if $F$ is a finite subset of $M(L)$ and $n\in \NN$, then 
$$W(F,\tfrac{1}{n})=\{f\in C_w(L):(\forall \mu\in F)\;\;\lvert \mu(f) \rvert<\tfrac{1}{n}\}$$
is a basic open neighborhood of the function equal to zero on $L$ in $C_w(L)$.
If $F=\{x\}$ or $F=\{\mu\}$ we will write $O(x,\tfrac{1}{m})$, $W(\mu,\tfrac{1}{m})$ rather than $O(\{x\},\tfrac{1}{m})$, $W(\{\mu\},\tfrac{1}{m})$.
If $\overline{x}\in K^k$ is a finite sequence of length $k$ consisting of elements of $K$ by $O(\overline{x},\tfrac{1}{m})$ we will mean the set
$O(F,\tfrac{1}{m})$, where $F$ is a set of elements of the sequence $\overline{x}$.

For $\mu\in M(L)$ and $n\in \NN$ we put
$$\overline{W}(\mu,\tfrac{1}{n})=\{f\in C_w(L): \lvert \mu(f) \rvert\leq \tfrac{1}{n}\}.$$

For a normed space $X$ by $B_X$ we denote the closed unit ball centered at $0$, i.e. $B_X=\{x\in X:\lVert x \rVert\leq 1\}$.

\section{Some simple observations}
Let us recall that the weak and the pointwise topology are related in the following way: For a compact space $K$, $C_w(K)$ is linearly homeomorphic to a closed linear subspace of $C_p(B_{M(K)})$ (here the unit ball $B_{M(K)}$ of the space of measures on $K$ is equipped with the weak* topology). An appropriate embedding is given by the standard evaluation map $i(f)(\mu) = \mu(f)$ for  $f\in C(K), \mu\in B_{M(K)}$.\medskip

In this section we will describe a few instances when the answer to Problem \ref{problem1} is immediate.

\begin{enumerate}[(A)]\itemsep5pt
 \item Let $K$ be an infinite countable compact space. Then $C_p(K)$ and $C_w(K)$ are not homeomorphic because in that case the pointwise topology
 is metrizable whereas the weak one is not.
 
 \item Suppose that $K$ is infinite and $|K| < |M(K)|$. 
 We have $|K|=\chi(C_p(K))$ and $\chi(C_w(K))=|M(K)|$, where $\chi(X)$ is a character of a topological space $X$ (see \cite[p. 13]{Tkachuk}).
 So in that case the character
 distinguishes between the pointwise and the weak topology. Let us remark that $|K| < |M(K)|$ if the cofinality of $|K|$ is countable.
 
 \item\footnote{This case was pointed out to us by Grzegorz Plebanek}
 Suppose that $K$ is a non-separable compact space such that there is a family $\{\mu_n\in M(K):n\in\omega\}$ of functionals separating
 elements of $C(K)$ (equivalently, there is a linear continuous injection $T:C(K)\to \ell_\infty$). Then $\psi(C_p(K))=d(K)>\omega$ 
 (see \cite[173]{Tkachuk}) and $\psi(C_w(K))=\omega$, where $\psi(X)$
 is a pseudocharacter of a space $X$ (see \cite[p. 13]{Tkachuk}) and $d(X)$ is its density (i.e. the minimal cardinality of a dense subset of $X$).
 
 The concrete example of a space having the property described here is $St(\mathcal{M})$, i.e., the Stone space of the measure algebra associated with the Lebesgue measure $\lambda$ on $[0,1]$. Indeed, it is well-known
 that $St(\mathcal{M})$ is not separable and $C(St(\mathcal{M}))\approx L_\infty([0,1])\approx \ell_\infty$.
 
 \item Let $K$ be an infinite scattered space. Then $C_p(K)$ is Fr\'echet-Urysohn and $C_w(K)$ is not. In Section 5 we shall examine this situation more closely.
\end{enumerate}

It is also not difficult to prove that $C_w(K)$ and $C_p(L)$ are never uniformly homeomorphic, provided $K$ and $L$ are infinite compacta.

\begin{fact} For any infinite compact spaces $K$ and $L$ the spaces $C_w(K)$ and $C_p(L)$ are not uniformly homeomorphic.
\end{fact}

\begin{proof} Assume towards a contradiction that, for some infinite compact spaces $K$ and $L$, there exists a uniform homeomorphism $\Phi: C_w(K) \to C_p(L)$. Without loss of generality we may assume that $\Phi(\underline{0}) = \underline{0}$.
From the uniform continuity of  $\Phi$ it follows that, for any $y\in L$, there exist a finite set $F_y\subseteq B_{M(K)}$ and $n_y\in \mathbb{N}$ such that
\begin{equation}\label{u1}
(\forall f,g\in C_w(K))\ (f-g)\in W(F_y,\tfrac{1}{n_y}) \Rightarrow (\Phi(f)-\Phi(g))\in O(y,1) .
\end{equation}
Then 
\begin{equation}\label{u2}
(\forall m\in \mathbb{N})\ \Phi(W(F_y,\tfrac{m}{n_y})) \subseteq O(y,m) .
\end{equation}
Indeed, for any $f\in W(F_y,\tfrac{m}{n_y})$ we take $f_k = \tfrac{k}{m}f$, for $k= 0,1,\dots,m$. Then $(f_{k+1} - f_k)\in W(F_y,\tfrac{1}{n_y})$,  therefore by
(\ref{u1}) $|\Phi(f_{k+1})(y) - \Phi(f_{k})(y)|< 1$ for any $k<m$. Since $\Phi(f_{0}) = \Phi(\underline{0}) = \underline{0}$, it follows that $|\Phi(f)(y)| = |\Phi(f_{m})(y)|< m$.

Since $L$ is infinite, we can find a sequence $(V_k)_{k\in \mathbb{N}}$ of nonempty, open pairwise disjoint subsets of $L$. For each $k\in \mathbb{N}$ pick a point $y_k\in V_k$, and take a continuous function $g_k: L \to [0,\infty)$ such that $g_k(y_k)= kn_{y_k}$ and $g_k^{-1}((0,\infty))\subseteq V_k$. Then the sequence $(g_k)_{k\in \mathbb{N}}$ converges to $\underline{0}$ in $C_p(L)$, hence the set $A = \{g_k: k\in \mathbb{N}\}\cup \{\underline{0}\}$ is compact in $C_p(L)$. Therefore $\Phi^{-1}(A)$ is compact in $C_w(K)$, so it is norm-bounded. On the other hand, by (\ref{u2}), we have  $\Phi^{-1}(g_k)\notin W(F_{y_k},\tfrac{kn_{y_k}}{n_{y_k}}) = W(F_{y_k},k)$, hence there is $\mu\in F_{y_k}$ such that $|\mu(\Phi^{-1}(g_k))|\ge k$. Since $\|\mu\|\le 1$, we have $\|\Phi^{-1}(g_k)\|\ge k$, a contradiction.
\end{proof}
One can easily verify that the above argument actually shows that for any infinite compact spaces $K$ and $L$ there is no homeomorphism $\Phi: C_w(K) \to C_p(L)$ which is uniformly continuous, i.e., we did not use the uniform continuity of $\Phi^{-1}$ in the proof.

\begin{col}
If $K$ and $L$ are infinite compact spaces,  then the spaces $C_w(K)$ and $C_p(L)$ are not linearly homeomorphic.
\end{col}

\section{Compacta containing closed uncountable metrizable subspaces}

It was proved in \cite{Kr} that if $K$ is an infinite metrizable finite-dimensional compactum, then $C_p(K)$ and $C_w(K)$ are not homeomorphic. In this
section we will extend this result to a certain class of non-metrizable compacta (cf. Corollary \ref{Valdivia}).
Recall that a normal space is \textit{strongly countable-dimensional} if it can be represented as a countable union of closed finite-dimensional subspaces.
In particular, any finite-dimensional space is
strongly countable-dimensional.
Let us prove the following.

\begin{thrm}\label{faktor}
Let $K$ be a compact strongly countable-dimensional space and let $L$ be a compact space such that $C_w(L)$ is homeomorphic to $C_w(M)\times E$ 
for some uncountable metrizable compact space $M$ and a topological space $E$. Then $C_p(K)$ and $C_w(L)$ are not homeomorphic.
\end{thrm}

\begin{proof}
By Miljutin's theorem \cite{Miljutin} (cf. \cite[4.4.8]{AK}) $C_w(M)$ is (linearly) homeomorphic to $C_w(Q)$, where $Q$ denotes the Hilbert cube.

Striving for a contradiction, let us assume that $C_p(K)$ and $C_w(L)$ are homeomorphic. Hence, from what we observed above, there is a homeomorphism
$\Phi:C_p(K)\to C_w(Q)\times E$. Without loss of generality we may assume that $\Phi(\underline{0})=(\underline{0},e)$, where  $e$ is a fixed point in $E$. For $k,m\in\mathbb{N}$, let
$$Z_{k,m}=\{(\overline{x},y)\in K^k\times Q: \Phi\left(O\left(\overline{x},\tfrac{1}{m}\right)\right)\subseteq \overline{W}(\delta_y,1)\times E\}.$$

One can easily show that for any $k,m\in\mathbb{N}$, the complement of $Z_{k,m}$ is open in $K^k\times Q$ and hence $Z_{k,m}$ is closed in $K^k\times Q$
(see \cite[Proposition 1.2]{Ok}).
By $\pi_1:K^k\times Q\to K^k$ and $\pi_2:K^k\times Q\to Q$ we denote the respective projections.

\medskip

\textbf{Claim 1.} For any $\overline{x}\in K^k$ the fiber $\pi_1^{-1}(\overline{x})\cap Z_{k,m}$ is at most countable (hence zero-dimensional).
\begin{proof}
Fix $\overline{x}\in K^k$. Since $\Phi:C_p(K)\to C_w(Q)\times E$ is a homeomorphism, there are a natural number $n\in\NN$, measures $\mu_1,\ldots ,
\mu_n\in M(Q)$ and an open set $U\subseteq E$ containing $e$ such that
\begin{equation}\label{(1)}
\Phi\left(O\left(\overline{x},\tfrac{1}{m}\right)\right)\supseteq W\left(\mu_1,\ldots ,\mu_n,\tfrac{1}{n}\right)\times U.
\end{equation}
Let $y\in \pi_1^{-1}(\overline{x})\cap Z_{k,m}$. We claim that $\delta_y\in\vspan(\mu_1,\ldots,\mu_n)$. Indeed, otherwise
$$\bigcap\{\Ker(\mu_i):i\leq n\}\nsubseteq \Ker(\delta_y),$$
where $\Ker(\nu)$ denotes the kernel of
a functional $\nu$ (cf. \cite[Lemma 3.9]{F}).
This means that
there is a continuous function $g:Q\to\mathbb{R}$ such that $\delta_y(g)=g(y)\neq 0$ and $\mu_i(g)=0$, for any $i\leq n$.
Scaling $g$ if necessary, we have
\begin{equation}\label{(2)}
\delta_y(g)=g(y)=2 \text{ and } g\upharpoonright \vspan(\mu_1,\ldots,\mu_n)=0.
\end{equation}

It follows from \eqref{(1)} that
$$\Phi^{-1}(g,e)\in O\left(\overline{x},\tfrac{1}{m}\right).$$
Since $y\in \pi_1^{-1}(\overline{x})\cap Z_{k,m}$, we have $$\Phi(\Phi^{-1}(g,e))=(g,e)\in \overline{W}(\delta_y,1)\times E.$$
Therefore $|\delta_y(g)|\leq 1$, contradicting \eqref{(2)}.

We have proved that if $y\in \pi_1^{-1}(\overline{x})\cap Z_{k,m}$, then $\delta_y\in\vspan(\mu_1,\ldots,\mu_n)$. However this means that such $y$ is an atom
of one of the measures $\mu_1,\ldots ,\mu_n$. Hence the set $\pi_1^{-1}(\overline{x})\cap Z_{k,m}$ is included in the countable set of atoms of measures
$\mu_1,\ldots ,\mu_n$.
\end{proof}

\medskip

\textbf{Claim 2.} For any $y\in Q$ the fiber $\pi_2^{-1}(y)\cap Z_{k,m}$ is finite.
\begin{proof}
Let $C(k,m)=\pi_2(Z_{k,m})$ and $E(k,m)=C(k,m)\setminus C(k-1,m)$ for $k>1$, $E(1,m)=C(1,m)$. To prove Claim 2 it is enough to show that for any $y\in E(k,m)$ the set
$\pi_2^{-1}(y)\cap Z_{k,m}$ is finite.

Striving for a contradiction, assume that $\pi_2^{-1}(y)\cap Z_{k,m}$ is infinite for some $y\in E(k,m)$. This implies that there are infinitely many
$k$-element sets $F_1,F_2,\ldots \subseteq K$ such that
\begin{equation}\label{(4)}
\Phi\left(O\left(F_i,\tfrac{1}{m}\right)\right)\subseteq \overline{W}(\delta_y,1)\times E, 
\end{equation}
for any $i\in \NN$. By a suitable version of $\Delta$-system lemma (see \cite[A.1.4]{vM}), we can assume that there is $A_0\subseteq K$ with $|A_0|\leq k-1$
and pairwise disjoint sets $A_1,A_2,\ldots \subseteq K$ such that $A_0\cup A_i=F_i$, for any $i\in \NN$.

We will prove that
\begin{equation}\label{(3)}
\Phi\left(O\left(A_0,\tfrac{1}{m}\right)\right)\subseteq \overline{W}(\delta_y,1)\times E 
\end{equation}
and this will be a desired contradiction since $|A_0|\leq k-1$ and $y\in E(k,m)=C(k,m)\setminus C(k-1,m)$ if $k>1$ (If $k=1$ we have
a contradiction with the surjectivity of $\Phi$).

If \eqref{(3)} does not hold, there is $f_0\in O\left(A_0,\tfrac{1}{m}\right)$ with $\Phi(f_0)=(f_1,f_2)$ such that $|\delta_y(f_1)|=|f_1(y)|>1$.
The set $$\Phi^{-1}\left(\{f\in C_w(Q): |\delta_y(f)|>1\}\times E\right)$$
is an open neighborhood of $f_0$, hence there is a finite set $B\subseteq K$ and $\varepsilon>0$ such that for $g\in C_p(K)$
\begin{equation}\label{(5)}
\lVert g\upharpoonright B - f_0\upharpoonright B \rVert<\varepsilon\; \text{ implies }\; \Phi(g)\in \{f\in C_w(Q): |\delta_y(f)|>1\}\times E,
\end{equation}
where $\lVert \cdot \rVert$ is the supremum norm.

Since $A_1,A_2,\ldots$ is an infinite collection of pairwise disjoint sets, there is $i\in\NN$ with $A_i\cap B=\emptyset$. Let $h\in C_p(K)$ be a function
such that $h\upharpoonright A_0\cup B=f_0$ and $h\upharpoonright A_i=0$.
Let $\Phi(h)=(h_1,h_2)$.
By \eqref{(4)}, $|\delta_y(h_1)|\leq 1$. But, on the
other hand $\delta_y(h_1)>1$ by \eqref{(5)}, a contradiction. We proved \eqref{(3)}, and as we have explained this ends the proof of Claim 2. 
\end{proof}
Since $K$ is strongly countable-dimensional compact, for any $k\in\NN$ the space $K^k$ is strongly countable-dimensional as well.
The set $Z_{k,m}$ is compact for any $k,m\in \NN$ and hence
the mapping $\pi_1\upharpoonright Z_{k,m}$ is closed.
By Claim 1 and \cite[5.4.7]{En2}), the set $Z_{k,m}$ is strongly countable-dimensional. Now, the mapping
$\pi_2\upharpoonright Z_{k,m}$ is closed and hence Claim 2 and \cite[5.4.A (d)]{En2}) imply that the set
$C(k,m)=\pi_2(Z_{k,m})$ is strongly countable-dimensional. By the continuity of $\Phi$ it follows that $Q=\bigcup\{C(k,m):k,m\in \NN\}$.
However $Q$ is strongly infinite-dimensional and thus cannot be a countable union of finite-dimensional subspaces. 
\end{proof}

\begin{col}\label{metr_podp}
If $K$ is a compact strongly countable-dimensional space and $L$ is a compact space containing a closed uncountable metrizable subspace,
then $C_p(K)$ and $C_w(L)$ are not homeomorphic.
\end{col}

\begin{proof}
Let $M\subseteq L$ be a closed uncountable metrizable subspace of $K$. It is well known, that $C_w(L)$ is (linearly) homeomorphic to
$C_w(M)\times \{f\in C_w(L): f\upharpoonright M =0\}$. This follows from the fact that if a compact space $M\subseteq L$ is metrizable,
then there exists a linear continuous extension operator $e:C(M)\to C(L)$ (see \cite[II.4.14]{LT}) which gives an isomorphism between
$C(L)$ and $C(M)\times  \{f\in C(L): f\upharpoonright M =0\}$ (see \cite[page 89]{AK}). 

The above factorization allows us to apply Theorem \ref{faktor}.
\end{proof}
Obviously, the assumption that the space $L$ contains a closed uncountable metrizable subspace is equivalent to the condition that $L$ contains a topological copy of the Cantor set.

Since any finite-dimensional space is strongly countable-dimensional we have the following.

\begin{col}\label{metr_podprz}
If $K$ is a compact finite-dimensional space and $L$ is a compact space containing a closed uncountable metrizable subspace, then $C_p(K)$ and $C_w(L)$
are not homeomorphic.
\end{col}

Given a set  $\Gamma$ we use the standard notation $\Sigma(\Gamma)$ for the $\Sigma$-product of real lines indexed by $\Gamma$, i.e., the subspace of the product $\mathbb{R}^\Gamma$ constisting of functions with countable supports. Let us recall that a compact space $K$ is called a Valdivia compact space if, for some set $\Gamma$, there exists an embedding $i: K\to \mathbb{R}^\Gamma$ such that the intersection $i(K)\cap \Sigma(\Gamma)$ is dense in $i(K)$. The following fact is probably known; since we could
not find a proper reference in the literature, we shall enclose a proof here. The argument
presented below was communicated to the authors by Grzegorz Plebanek.

\begin{fact}\label{Valdivia_dychotomia} 
Every Valdivia compact space is either scattered or contains a closed uncountable metrizable subspace.
\end{fact}

\begin{proof} Let $K$ be a Valdivia compact space. Without loss of generality we can assume that $K$ is a subset of the product $\mathbb{R}^\Gamma$, such that $K\cap \Sigma(\Gamma)$ is dense in $K$.
If  $K$ is non-scattered, then there exists a continuous surjection $\varphi: K\to [0,1]$. Let $\Phi: \mathbb{R}^\Gamma \to [0,1]$ be a continuous extension of $\varphi$ over $\mathbb{R}^\Gamma$ (cf. \cite[p.\ 368]{vM}). By \cite[Problem 2.7.12(c)]{En1} there is a countable set $A\subset \Gamma$ such that
\begin{equation}\label{Valdivia_dychotomia1}
\forall (x,y \in \mathbb{R}^\Gamma)\ x\upharpoonright A = y\upharpoonright A \Rightarrow \Phi(x) = \Phi(y). 
\end{equation}
For any subset $J$ of $\Gamma$ consider the map $r_J: \mathbb{R}^\Gamma\to \mathbb{R}^\Gamma$ given by
\begin{equation*}
r_J(x)(\gamma)= \begin{cases}
x(\gamma)& \mbox{ for } \gamma\in J,\\
0& \mbox{ for } \gamma\in \Gamma\setminus J.
\end{cases}
\end{equation*}
From \cite[Lemma 1.2]{AMN} it follows that there is a countable set $B$ such that $A\subseteq B\subseteq \Gamma$ and $r_B(K)\subseteq K$. Using (\ref{Valdivia_dychotomia1}) one can easily verify that $\varphi = \varphi\circ r_B$. Therefore $\varphi(r_B(K))= [0,1]$, hence $r_B(K)$ is uncountable. Obviously, $r_B(K)$ is compact and metrizable, since $B$ is countable.
\end{proof}
Since every infinite scattered compact space contains a nontrivial convergent sequence, from the above proposition easily follows the well-known fact that each infinite Valdivia compact space contains a nontrivial convergent sequence cf. \cite[Theorem 3.1.1]{Ka}.

\begin{col}\label{Valdivia}
If $K$ is an infinite finite-dimensional Valdivia compact space, then $C_p(K)$ and $C_w(K)$ are not homeomorphic.
\end{col}

\begin{proof}
If $K$ is not scattered then by Theorem \ref{Valdivia_dychotomia} $K$ contains a closed uncountable metrizable subspace. By Corollary \ref{metr_podprz} we are done.
The case when $K$ is scattered is covered by Corollary \ref{wniosek1} below.
\end{proof}

Let us recall that the double arrow space $\mathbb{K}$ is the set $\mathbb{K}=((0,1]\times \{0\})\cup ([0,1)\times\{1\})$
equipped with the order topology given by the
lexicographical order (i.e., $(s,i)\prec(t,j)$ if either $s<t$, or
$s=t$ and $i<j$).

\begin{fact}\label{arrow}
For the double arrow space $\mathbb{K}$, the function spaces $C_p(\mathbb{K})$ and $C_w(\mathbb{K})$ are not homeomorphic.
\end{fact}
\begin{proof}
It was proved in \cite[Lemma 4.6]{M} that for each nonempty compact metrizable space $M$, the spaces $C(\mathbb{K})$ and $C(\mathbb{K})\times C(M)$ are isomorphic.
Hence $C_w(\mathbb{K})$ is homeomorphic to $C_w(\mathbb{K})\times C_w([0,1])$ and we can apply Theorem \ref{faktor}.
\end{proof}

Recall that the double arrow space is not scattered, but any metrizable subspace of $\mathbb{K}$ is countable. Therefore in the above proof we cannot use Corollary \ref{metr_podprz} instead of Theorem \ref{faktor}.

Problems \ref{problem1} and \ref{problem2} remain open in their full generality.
But there are also concrete spaces for which our methods do not work. Perhaps the most natural particular instances of Problem \ref{problem1}
which remain open are the following questions:

\begin{question}(see \cite[Question 2]{Kr})
 Is it true that $C_p([0,1]^\omega)$ and $C_w([0,1]^\omega)$ are not homeomorphic?
\end{question}

\begin{question}
 Is it true that $C_p(\beta\omega)$ and $C_w(\beta\omega)$ are not homeomorphic?
\end{question}

\begin{question}
 Is it true that $C_p(\beta\omega\setminus\omega)$ and $C_w(\beta\omega\setminus\omega)$ are not homeomorphic?
\end{question}
Here $\beta\omega$ denotes the \v{C}ech-Stone compactification of the space of natural numbers $\omega$.

\section{Property (B), Fr\'echet-Urysohn spaces and scattered compacta}\label{property_B}

The following definition was suggested to us by T.\ Banakh.

\begin{defin}\label{def}
A space $X$ has the property (B) provided $X$ can be covered by countably many closed nowhere-dense sets $\{A_n:n\in\omega\}$ such that
for any compact set $K\subseteq X$ there exists $n\in \omega$ with $K\subseteq A_n$.
\end{defin}

A family $\mathcal{A}$ of subsets of a space $X$ such that any compact subspace of $X$ is contained in some member of $\mathcal{A}$ is sometimes called a \emph{$k$-cover}.

A large class of spaces having the property (B) is formed by all infinite-dimensional Banach spaces endowed with the weak topology.
Indeed, if $X$ is an infinite-dimensional Banach space we can simply take as
$A_n$ the $n$-ball, i.e. $A_n=nB_X=\{x\in X:\| x \|\leq n \}$. Now, if $K$ is a weakly compact subset of $X$, then $K$ is norm bounded and hence $K\subseteq A_n$,
for some $n$. Moreover, each $A_n$ is weakly closed and has empty interior (in the weak topology), since all non-empty weakly open sets in the
infinite-dimensional Banach space $X$ are
not bounded.

Actually, using the same argument one can easily obtain a more general fact.

\begin{fact}\label{B_weak_general}
Let $(X,\|\cdot\|)$ be a normed space and let $\tau$ be a linear topology on $X$, strictly weaker than the norm topology. 
If norm closed balls in $X$ are $\tau$-closed and $\tau$-compact sets are norm bounded, then $(X,\tau)$ has the property (B).

In particular, for an infinite-dimensional Banach space $X$, both spaces $(X,w)$ and $(X^*,w^*)$ possess the property (B).
\end{fact}

The next theorem relates the property (B) and Fr\'echet-Urysohn property in the setting of topological spaces.

\begin{thrm}\label{no_B_and_FU}
If $X$ is a nonempty Fr\'echet-Urysohn topological space, then $X$ does not have the property (B).
\end{thrm}
\begin{proof}
Assume towards a contradiction that $X$ is a nonempty Fr\'echet-Urysohn space with the property (B). Let $\{A_n:n\in\omega\}$ be a sequence of subsets of
$X$ witnessing the property (B), we can additionally assume that this sequence is increasing. 

Property (B) implies that $X$ has no isolated points. Fix a point $x\in X$; by Fr\'echet-Urysohn property we can find a sequence $(x_n)_{n\in\omega}$
of distinct points of $X$ converging to $x$. We can also assume that every $x_n$ is distinct from $x$.
For every $n$, since $A_n$ is nowhere-dense, we can find a sequence $(x^k_n)_{k\in\omega}$ of distinct points from $X\setminus A_n$ converging to $x_n$. Let
$$S=\{x^k_n: k,n\in\omega\}\setminus(\{x_n: n\in\omega\}\cup\{x\}).$$
Observe that, for every $n$, the sequences $(x_n)_{n\in\omega}$ and $(x^k_n)_{k\in\omega}$ have different limits, therefore $S$ contains all but finitely many
elements of $(x^k_n)_{k\in\omega}$. It follows that all points $x_n$ are in the closure of $S$, and consequently this closure contains also the point $x$.
Since $X$ is Fr\'echet-Urysohn, we can find a sequence $(y_i)_{i\in\omega}$ of points of $S$ converging to $x$.
Consider the compact set $K= \{y_i: i\in\omega\}\cup\{x\}$. For every $n$, the set $K$ can contains only finitely many elements of the sequence
$(x^k_n)_{k\in\omega}$, therefore it must contain elements from infinitely many such sequences. If $x^k_n\in K$, then $K$ is not contained in $A_n$,
so $K$ is not contained in infinitely many $A_n$. Since the sequence $(A_n)_{n\in\omega}$ is increasing, it follows that no $A_n$ contains $K$, a contradiction.
\end{proof}

From the last two results it follows that, for any infinite compact space $K$, the space $C_w(K)$ has the property (B) and is not Fr\'echet-Urysohn
(see \cite{SW} for a more general result). A function space equipped with the pointwise topology may not
have the property (B) (cf. Theorem \ref{B_charakteryzacja} below). However
we have the following.
\begin{fact}\label{B_przelicz}
Let $K$ be a compact space with a countable family $\mathcal{S}$ of infinite subsets, such that any nonempty open subset of $K$ contains a member of $\mathcal{S}$.
Then $C_p(K)$ has the property (B). 
\end{fact}
\begin{proof}
For $n\in\NN$ and $S\in \mathcal{S}$ put $$A_{n,S}=\{f\in C_p(K):f(S)\subseteq [-n,n]\}.$$
We will show that $\{A_{n,S}:n\in\NN,S\in\mathcal{S}\}$ is a countable collection of closed nowhere-dense sets witnessing the property (B) for $C_p(K)$.

Obviously, for each $n\in\NN$ and $S\in \mathcal{S}$, the set $A_{n,S}$ is closed in $C_p(K)$. It also has empty interior in $C_p(K)$ because each $S$ is infinite.
Now, take an arbitrary
compact set $A\subseteq C_p(K)$. The set $A$ is pointwise bounded being compact in the pointwise topology. Thus, we have
$$K=\bigcup_{n\in\NN}\{x\in K: (\forall f\in A)\;f(x)\in[-n,n]\}.$$
Observe that the above union consists of closed sets.
By the Baire category theorem, there is $n_0\in \NN$ such that the set $$K_0=\{x\in K: (\forall f\in A)\;f(x)\in[-n_0,n_0]\}$$ has non-empty interior in $K$.
It follows that there is $S\in\mathcal{S}$ with $S\subseteq K_0$ and hence $A\subseteq A_{n_0,S}$.
\end{proof}

Recall that a family $\mathcal{B}$ consisting of nonempty open subsets of a topological space $X$ is a \textit{$\pi$-base} if for any
nonempty open set $U\subseteq X$, there
is $B\in\mathcal{B}$ such that $B\subseteq U$.

\begin{col}\label{B_pi_baza}
If $K$ is a dense-in-itself compact space with a countable $\pi$-base, then $C_p(K)$ has the property (B).
\end{col}

A surjective map $f:X\to Y$ between topological spaces is said to be \textit{irreducible} if no proper closed subset of $X$ maps onto $Y$.
If $X$ is compact, by Kuratowski-Zorn Lemma,
for any surjective map $f:X\to Y$ there is a closed subset $C\subseteq X$ such that the restriction $f\upharpoonright C$ is irreducible.

\begin{fact}\label{B_nieprzywiedlne}
Every non-scattered compact space contains a dense-in-itself compact subspace with a countable $\pi$-base.
\end{fact}

\begin{proof}
Let $K$ be a non-scattered compact space. Fix a continuous surjection $\varphi:K\to [0,1]$. There is a closed subset $C\subseteq K$ such that the
mapping $\varphi\upharpoonright C:C\to [0,1]$ is irreducible.
We claim that $C\subseteq K$ is a compact subspace we are looking for. Indeed,
since $[0,1]$ is dense-in-itself and $\varphi\upharpoonright C$ is irreducible, it follows that
$C$ is dense-in-itself as well. Similarly, irreducibility of a closed map $\varphi\upharpoonright C$ and the existence of countable $\pi$-base in $[0,1]$
imply the existence
of a countable $\pi$-base in $C$ (see \cite[S.228, Fact 1]{Tkachuk}). 
\end{proof}

\begin{fact}\label{B_podprz}
If a Tychonoff space $X$ contains a compact subspace $K$ such that $C_p(K)$ has the property (B), then $C_p(X)$ also has the property (B). 
\end{fact}

\begin{proof}
Consider the restriction map $p:C_p(X)\to C_p(K)$ defined by $p(f)=f\upharpoonright K$. Note that compactness of $K$ imply that $p$ is open and onto.
Let $\{A_n:n\in\omega\}$ be a family of closed nowhere-dense subsets of $C_p(K)$ witnessing the property (B) for $C_p(K)$. Since $p$ is open, for each $n\in\omega$,
the set $p^{-1}(A_n)$ is a (closed) nowhere-dense subset of $C_p(X)$. Take an arbitrary compact set $A\subseteq C_p(X)$. By continuity of $p$, the set $p(A)$ is
compact, so $p(A)\subseteq A_n$, for some $n\in\omega$. This gives $A\subseteq p^{-1}(A_n)$.
\end{proof}

Corollary \ref{B_pi_baza}, Propositions \ref{B_nieprzywiedlne} and \ref{B_podprz} immediately imply the following

\begin{col}\label{B_nieroprosz}
For any  non-scattered compact space $K$, the space $C_p(K)$ has the property (B).
\end{col}

\begin{thrm}\label{B_charakteryzacja}
Let $K$ be a compact space. The following conditions are equivalent
\begin{enumerate}
 \item[(a)] $K$ is scattered,
 \item[(b)] $C_p(K)$ is Fr\'echet-Urysohn,
 \item[(c)] $C_p(K)$ does not have the property (B).
\end{enumerate}
\end{thrm}
\begin{proof}
Equivalence $(a)\Leftrightarrow (b)$ is due to Gerlits and Pytkeev (cf. \cite[III.1.2]{Arha}). Theorem \ref{no_B_and_FU} provides implication $(b)\Rightarrow (c)$.
From Corollary \ref{B_nieroprosz} it follows that $(c)$ implies $(a)$.
\end{proof}

Let us remark that for Tychonoff spaces $X$, the Fr\'echet-Urysohn property of $C_p(X)$ is not equivalent to the failure of the property (B).
To see this consider $\omega_1$ with the discrete
topology. Then $C_p(\omega_1)=\RR^{\omega_1}$ does not have the property (B) being a Baire space and is not Fr\'echet-Urysohn. This motivates the following problem.
\begin{prob}
Characterize the property (B) of $C_p(X)$ in terms of the topology of a Tychonoff space $X$.
\end{prob}

\begin{col}\label{wniosek1}
Let $K$ be an infinite compact space and let $S$ be an infinite scattered compact space. The spaces $C_w(K)$ and $C_p(S)$ are not homeomorphic. 
\end{col}
\begin{proof}
By Theorem \ref{B_charakteryzacja}, the space $C_p(S)$ does not have the property (B). On the other hand, as we already observed $C_w(K)$ has the property (B)
(see remarks following Theorem \ref{no_B_and_FU}).
\end{proof}

We can also prove the following analogous result going in the ``opposite direction''.
\begin{thrm}\label{thm-scattered-weak}
Let $K$ be an infinite compact space and let $S$ be an infinite scattered compact space. The spaces $C_p(K)$ and $C_w(S)$ are not homeomorphic. 
\end{thrm}

\begin{proof}
Striving for a contradiction, assume that there is a homeomorphism $\varphi:C_p(K)\to C_w(S)$. By \ref{wniosek1}, the space $K$ is not scattered.
Hence $K$ can be continuously mapped onto the closed unit interval $[0,1]$ and in effect $C_p([0,1])$ embeds into $C_p(K)$. Denote by $Y$ a
homeomorphic copy of $C_p([0,1])$ in $C_p(K)$. Consider
$$X=\overline{\vspan(\varphi(Y))}^{w}=\overline{\vspan(\varphi(Y))}^{\| \cdot \|}.$$
The space $Y$ is separable being a copy of $C_p([0,1])$ and thus $X$ is norm-separable. Since $S$ is scattered, the space $C(S)$ is Asplund
\cite[Theorem 12.29]{F}
and hence $X^*$ is separable \cite[Theorem 8.26]{F}. This implies that $B_X$ in its weak topology is metrizable \cite[Proposition 3.28]{F}.
So $X=\bigcup_{n\in \omega}nB_X$ is a countable union of metrizable spaces.
Transferring this property by the homeomorphism $\varphi$ we see that $Y$ (and hence $C_p([0,1])$) is a countable union of metrizable spaces, which is impossible
(see \cite[Problem 446]{Tkachuk2} and \cite[Problem 210]{Tkachuk}).
\end{proof}

Another property which topologically distinguishes spaces $C_p(K)$ and $C_w(K)$ for infinite scattered compact $K$ is the Fr\'echet-Urysohn property.
Equivalence $(a)\Leftrightarrow (b)$ in Theorem \ref{B_charakteryzacja} says that, a compact space $K$ is a scattered if and only if $C_p(K)$ is Fr\'echet-Urysohn.
On the other hand, if $K$ is infinite, the space
$C_w(K)$ is not Fr\'echet-Urysohn (see remarks following Theorem \ref{no_B_and_FU}). 

\begin{remark}\label{rem-gen}
 It was proved in \cite[Corollary 1]{Kr} that if $K$ is an infinite compact metrizable $C$-space then $C_p(K)$ and $C_w(K)$ are not homeomorphic.
 This settled Problem
 \ref{problem1} for metrizable $C$-spaces. In fact we can say more (cf. Problem \ref{problem2}) in this situation. Namely, if $K$ is infinite compact
 metrizable $C$-space and $L$ is an arbitrary compactum, then $C_p(K)$ and $C_w(L)$ are not homeomorphic.
\end{remark}

 \begin{proof}
 First, observe that $L$ cannot be non-metrizable.
 Indeed, $K$ is compact metrizable so the network weight $nw(K)$ is countable. We have
 $\omega=nw(K)=nw(C_p(K))$ (see \cite[Problem 172]{Tkachuk}). Now, if $C_p(K)$ and $C_w(L)$ were homeomorphic, then $nw(C_w(L))=\omega$. But the pointwise topology
 is weaker then the weak topology, hence $\omega=nw(C_p(L))=nw(L)$ and we infer that $L$ is metrizable. From Theorem \ref{thm-scattered-weak} it follows that $L$
 is not scattered so being compact it must be uncountable. Note, that $K$ has to be uncountable too (otherwise, being metrizable, it would be scattered which
 is not possible by Corollary \ref{wniosek1}).
 Hence from Miljutin's theorem (see \cite[4.4.8]{AK}), $C_w(L)$ is homeomorphic to $C_w(K)$. However, as we mentioned, it was
 shown in \cite[Corollary 1]{Kr} that $C_p(K)$ and $C_w(K)$ are not homeomorphic.
\end{proof}

\section{Further remarks on the property (B)}\label{B_remarks}

Recall that a topological space $X$ is \emph{sequential} if every subset $A$ of $X$ containing limits of all convergent sequences from $A$ is closed in $X$.
A space $X$ is called a \emph{$k$-space} if every subset $A$ of $X$ is closed provided its intersection $A\cap K$ with any compact subset $K$ of $X$ is closed.
Clearly, every Fr\'echet-Urysohn space is sequential, and every sequential space is a $k$-space.

\begin{remark}\label{remark-sequential}
Let $X$ be a sequential space. Suppose that there is a collection $\{A_n:n\in\omega\}$ of closed nowhere-dense sets in $X$ such
that for any sequence $(x_i)_{i\in\omega}$ convergent in $X$ to $x$, there is $n\in \omega$ with $\{x_i:i\in\omega\}\cup\{x\}\subseteq A_n$.
Then $X$ has the property (B).
In other words,
for sequential spaces,
instead of checking the property (B) on arbitrary compact set, one can verify it on convergent sequences only.
\end{remark}
\begin{proof}
Without loss of generality we may assume that $A_n\subseteq A_{n+1}$, for any $n\in\omega$.
Let $K\subseteq X$ be compact. Suppose that $K\nsubseteq A_n$, for any $n\in \omega$. Then we can inductively construct 
an increasing sequence $(n_k)_{k\in\omega}$ of natural numbers and distinct points $x_k\in K\setminus A_{n_k}$. Indeed, put $n_0=0$ and
take $x_0\in K\setminus A_0$ (we assumed that $K\setminus A_n\neq \emptyset$). Suppose that $n_k$ and $x_k$ are already constructed.
By our assumption, there is $n_{k+1}>n_k$ with $A_{n_{k+1}}\supseteq \{x_0,\ldots , x_k\}$. Since $K\setminus A_{n_{k+1}}\neq \emptyset$, we can
pick $x_{k+1}\in K\setminus A_{n_{k+1}}$. This ends the inductive construction.

Since $K$ is compact, the set $Y=\{x_k:k\in\omega\}\subseteq K$ has an accumulation point $x\in K$.
If $x\notin Y$, then $Y$ is not closed and by sequentiality
there is a sequence $(x_{k_m})_{m\in \omega}$ converging to some point $z\notin Y$. Clearly, $\{x_{k_m}:m\in\omega\}\cup\{z\}\nsubseteq A_n$,
for any $n\in \omega$, contradicting
our assumption.

If $x\in Y$, then $x$ is an accumulation point of $Y'=Y\setminus\{x\}$ and $x\notin Y'$ so we can proceed as above.
\end{proof}

In Section \ref{property_B} we showed that the property (B) cannot be combined with Fr\'echet-Urysohn property. The next two examples demonstrate that
sequential spaces can have the property (B). The second one is a topological vector space.

\begin{exa}\label{S_omega}
Recall that the Arhangel'skii-Franklin $S_\omega$ space is the set $\omega^{<\omega}$, of all finite sequences of natural numbers, equipped with the
following topology
$\tau$:
$$U\in\tau\quad\Leftrightarrow\quad \forall s\in U\;\; \{n\in\omega:s\;\widehat{}\; n\notin U\}\;\text{is finite}$$
The space $S_\omega$ is sequential (see \cite{AF}). Let us prove that it has the property (B). For $n\in\omega$, put $A_n=\{s\in S_\omega: length(s)\leq n\}$.
Obviously, the set $A_n$ is closed and has empty interior.
It is not difficult to show (see \cite{AF}) that if $(s_i)$ is a sequence in $S_\omega$ converging
to $s\in S_\omega$, then there are $n_i\in\omega$ such that
$s_i=s\;\widehat{}\; n_i$, for all but finitely many $i\in\omega$.
Let $S=(s_i)$ be a convergent sequence in $S_\omega$. Denote its limit point by $s$. From what we noted above, it follows that
the set $I=\{i\in\omega:s_i\notin A_{lenght(s)+1}\}$ is finite. Let $m=\max\{length(s_i):i\in I\}$. Then $\{s_i:i\in\omega\}\cup\{s\}\subseteq A_m$.
By Remark \ref{remark-sequential}, the space $S_\omega$ has the property (B).
\end{exa}

\begin{exa}\label{Mackey}
Recall that the Mackey topology $\tau$ on the dual $X^*$ of a Banach space $X$ is the topology of uniform convergence on weak compact subsets of $X$.
This topology is weaker than the norm one and finer than the weak$^*$ topology.
By \cite[Theorem 5.6]{SW} it follows that the Banach space $\ell_\infty = \ell_1^*$ equipped with the  Mackey topology $\tau$ is a $k$-space.
For any compact subspace $K$ of $(\ell_\infty,\tau)$, the topology $\tau$ coincides with the weak$^*$ topology, therefore $(K,\tau)$ is metrizable,
since $\ell_1$ is separable. It easily follows that the space  $(\ell_\infty,\tau)$ is sequential.

Observe that norm closed balls in $\ell_\infty$ are weak$^*$ closed, hence also $\tau$-closed; and $\tau$-compact sets are weak$^*$ compact,
therefore norm bounded. Since $\tau$ is strictly weaker than the norm topology on $\ell_\infty$, from Proposition \ref{B_weak_general}
it follows that the space $(\ell_\infty,\tau)$ has the property (B).
\end{exa}

It seems that examples of topological vector spaces which are both sequential and with the property (B) are rather rare. If $X$ is an infinite-dimensional Banach space, then neither $(X,w)$  nor $(X^*,w^*)$ is a $k$-space (cf. \cite[p.\ 280]{SW} and \cite[p.\ 390]{KS}). Bellow we give a simple, elementary proof of these facts.

\begin{fact}\label{dual_pair} 
Let $\langle E,F\rangle$ be a pair of infinite-dimensional Banach spaces in duality, i.e., either $\langle E,F\rangle=\langle X,X^*\rangle$ or  $\langle E,F\rangle=\langle X^*,X\rangle$ for some infinite-dimensional Banach space $X$. Then $E$ contains a subset $A$ such that every bounded subset of $A$ is finite and $0\in \overline{A}^{\sigma(E,F)}\setminus A$.
\end{fact}
\begin{proof} For every positive integer $n$ take an $n$-dimensional subspace $E_n$ of $E$ and let $D_n$ be a finite $(1/n^2)$-dense subset of $E_n\cap S_E$. Define $A= \bigcup_n nD_n$. 

We shall check that $0\in \overline{A}^{\sigma(E,F)}$. Fix a basic $\sigma(E,F)$-open neighborhood $U$ of $0$ of the form $\{e\in E: |\langle e,f_i\rangle|<\varepsilon\mbox{ for }  i=0,1,\dots,k\}$, where $f_i\in S_F$ and $\varepsilon>0$. Take $n$ such that $n>k$ and $1/n <\varepsilon$. Since $\dim E_n > k$, we can find $e\in E_n\cap S_E$ such that $\langle e,f_i\rangle=0$ for $i=0,1,\dots,k$. Pick $h\in D_n$ such that $\|e-h\|\le 1/n^2$. Then $|\langle h,f_i\rangle|\le 1/n^2$, hence $|\langle nh,f_i\rangle|\le 1/n<\varepsilon$ for $i=0,1,\dots,k$. Therefore $nh\in A\cap U$.
\end{proof}

\begin{col}\label{not_k_space}
If $X$ is an infinite-dimensional Banach space, then neither $(X,w)$ nor $(X^*,w^*)$ is a $k$-space.
\end{col}

In \cite[Theorem 1.5]{GKP} it was proved that for $X$ as in the above corollary, the space $(X,w)$ is not even an \emph{Ascoli} space
(which is weaker than being a $k$-space). Using Proposition 2.1 from this paper and above Propoition \ref{dual_pair} 
one can easily prove a counterpart of this result for the weak$^*$ topology.

\begin{thrm}
The dual $X^*$ of a Banach space $X$ is Ascoli
in the weak$^*$ topology  if and only if
$X$ is finite-dimensional.
\end{thrm}


\begin{thebibliography}{AMN}

\bibitem[AK]{AK}
F.\ Albiac, N.J.\ Kalton {\em Topics in Banach Space Theory},
Graduate Texts in Mathematics 233, Springer, New York, 2006.

\bibitem[AMN]{AMN} S.\ Argyros, S.\ Mercourakis and S.\ Negrepontis, {\em Functional analytic properties of Corson compact spaces}, Studia Math. 89 (1988), 197--29.

\bibitem[Ar]{Arha}
A.V.\ Arhangel'skii, {\em Topological Function Spaces.}
Mathematics and its Applications (Soviet Series), 78. Kluwer Academic Publishers Group, Dordrecht, 1992.

\bibitem[AF]{AF}
A.V.\ Arhangel'skii, S.P.\ Franklin, {\em Ordinal invariants for topological spaces}, Michigan Math. J. 15 (1968), 313--320.


\bibitem[En1]{En1} R.\ Engelking, {\em General Topology}, PWN—Polish Scientific Publishers, Warsaw, 1977.

 \bibitem[En2]{En2}
R.\ Engelking, {\em Theory of Dimensions Finite and Infinite}, Sigma Series in Pure Mathematics, 10. Heldermann Verlag, Lemgo, 1995.

\bibitem[FH]{F}
M.\ Fabian, P.\ Habala, P.\ H\'ajek, V.\ Montesinos~Santalucía, J.\ Pelant, V.\ Zizler,
{\em Functional Analysis and Infinite-Dimensional Geometry}
CMS Books in Mathematics/Ouvrages de Math\'ematiques de la SMC, 8. Springer-Verlag, New York, 2001.

\bibitem[GKP]{GKP} S.\ Gabriyelyan, J.\ Kąkol, and G.\ Plebanek, {\em The Ascoli property for function spaces and the weak topology of Banach and Fréchet spaces}, Studia Mathematica 233 (2016), 119--139.

\bibitem[GN]{GN} J.\ Gerlits, Zs.\ Nagy, {\em Some properties of C(X). I},
Topology Appl. 14 (1982), no. 2, 151--161.

\bibitem[Ka]{Ka} O.\ Kalenda, {\em Valdivia compact spaces in topology and Banach space theory}, Extracta Math. 15 (1) (2000) 1--85.

\bibitem[KS]{KS} J.\ Kąkol,  S.A.\ Saxon, {\em Montel (DF)-spaces, sequential (LM)-spaces and the strongest locally convex topology}, J.\ London Math. Soc. 66 (2002), 388--406.

\bibitem[Kr1]{Kr1} M.\ Krupski, {\em On the $t$-equivalence relation}, Topology Appl. 160 (2013), no. 2, 368--373.

\bibitem[Kr2]{Kr} M.\ Krupski, {\em On the weak and pointwise topologies in function spaces}, Rev. R. Acad. Cienc. Exactas Fís. Nat. Ser. A Math. RACSAM, 110 (2016), 557--563.

\bibitem[LT]{LT}  J. Lindenstrauss and L. Tzafriri, \emph{Classical Banach spaces.} Lecture Notes in Mathematics, Vol. 338. Springer-Verlag, Berlin-New York, 1973.

\bibitem[Ma]{M} W.\ Marciszewski, {\em Modifications of the double arrow space and related Banach spaces $C(K)$},
Studia Math. 184 (2008), no. 3, 249–262. 

\bibitem[Mi]{Miljutin}
A.A. Miljutin, {\em Isomorphism of the spaces of continuous functions over compact sets of the cardinality of the continuum} (Russian)
Teor. Funkcii Funkcional. Anal. i Prilo\v{z}en. Vyp. 2 (1966) 150--156.

\bibitem[vM]{vM} J.\ van\ Mill, \textit{The Infinite-Dimensional Topology of Function Spaces}, North-Holland Mathematical Library 64, North-Holland, Amsterdam, 2001.

\bibitem[Ok]{Ok} O.\ Okunev, {\em A relation between spaces implied by their t-equivalence}, Topology Appl. 158 (2011), 2158--2164.

\bibitem[SW]{SW} G. Schl\"{u}chterman, R.F.\ Wheeler, {\em The Mackey dual of a Banach space}, Noti de Matematica, XI (1991), 273--287.

\bibitem[Tk1]{Tkachuk} V.V.\ Tkachuk, {\em A $C_p$-Theory Problem Book. Topological and Function Spaces.} Problem Books in Mathematics. Springer, New York, 2011.

\bibitem[Tk2]{Tkachuk2} V.V.\ Tkachuk {\em A $C_p$-Theory Problem Book. Special Features of Function Spaces.} Problem Books in Mathematics. Springer, Cham, 2014.
\end{thebibliography}
\end{document}